\documentclass[12pt,reqno,twoside]{amsart}
\usepackage{amsmath}
\usepackage{amssymb}

\usepackage{color}
\newcommand{\todo}[1]{{\color{blue} \sf To Do: [#1]}}

\newcommand{\ignore}[1] {}

\newcommand{\vre}{\varepsilon}

\newcommand\vare{\varepsilon}

\newcommand\NN{\mathbb N}
\newcommand\RR{\mathbb R}
\newcommand\ZZ{\mathbb Z}

\newcommand\TT{\mathbb T}

\newcommand\FF{\mathbb F}

\newcommand\cH{\mathcal{H}}

\newcommand\supp{\operatorname{supp}}

\newcommand\abs[1]{\left|#1\right|}

\newcommand\set[1]{\left\{{#1}\right\}}

\newtheorem{theorem}{Theorem}[section]
\newtheorem{lemma}[theorem]{Lemma}

\newtheorem{cor}[theorem]{Corollary}

\theoremstyle{definition}
\newtheorem{definition}[theorem]{Definition}

\theoremstyle{remark}
\newtheorem{remark}[theorem]{Remark}

\numberwithin{equation}{section}

\title[Effective equidistribution]{On Arnold's and Kazhdan's  \\ equidistribution problems}

\author{Alexander Gorodnik}
\address{School of Mathematics \\ University of Bristol \\ Bristol, U.K.}
\email{a.gorodnik@bristol.ac.uk}
\thanks{The first author was supported in part by EPSRC, ERC, and RCUK}

\author{Amos Nevo}
\address{Department of Mathematics, Technion}
\email{anevo@tx.technion.ac.il}
\thanks{The second author was supported by ISF grant}

\date{\today}

\begin{document}

\begin{abstract}
We consider isometric actions of lattices in semisimple algebraic groups on 
(possibly non-compact) homogeneous spaces with (possibly infinite)
invariant Radon measure. We assume that the action has a dense orbit, 
and demonstrate two novel and non-classical dynamical phenomena 
that  arise in this context. The first is the existence of a mean ergodic theorem even when the invariant measure is infinite,  
which implies the existence of an associated limiting distribution, possibly  different  than the invariant measure. 
The second is uniform quantitative equidistribution of all orbits in the space, which follows from a
quantitative mean ergodic theorem for such actions. In turn, these results imply quantitative ratio
ergodic theorems for isometric actions of lattices.  This sheds some unexpected light on certain equidistribution  problems posed by Arnol'd \cite{A} and also on the equidistribution conjecture for dense
subgroups of isometries formulated by Kazhdan \cite{K}.  We briefly describe the
general problem regarding ergodic theorems for actions of lattices on homogeneous spaces and its solution via the duality principle \cite{GN2}, and  give a number of examples to demonstrate our results. Finally, we also prove results on quantitative equidistribution for transitive actions.
\end{abstract}

\maketitle

\section{Equidistribution beyond amenable groups}

Let $G$ be a locally compact second countable (lcsc) group acting continuously
on an lcsc space $X$. 
Assume that $X$ carries a $\sigma$-finite $G$-invariant Radon measure $\mu$ of full support. 
Consider the following three natural conditions that often arise in practice. 
\begin{enumerate}
\item The action of $G$ on $X$ is transitive. 
\item The action of $G$ on $X$ has a unique invariant Radon measure.
\item The action of $G$ on $X$ is isometric.  
\end{enumerate}
Letting $\Gamma$ be a countable dense subgroup of $G$, consider the problem of formulating and establishing equidistribution results for the  $\Gamma$-orbits in $X$.  
 This problem has been studied in the past mainly in the case when the group $\Gamma$ is amenable and the measure $\mu$ is finite. It is a 
 compelling challenge to generalize the theory to the case  where the the group is non-amenable and the measure may be infinite. In particular, one would like to establish that a limiting distribution for the $\Gamma$-orbits exists, and study what its properties might be. 
 This challenge can be generalized further to the case where $X$ is any standard Borel space with $\sigma$-finite measure, where now the results sought are mean ergodic theorems in $L^p$ and pointwise ergodic theorems holding almost everywhere. 
 
 The possible choices of $G$, $X$ and $\Gamma$ include a large set of important  examples arising naturally in various branches of dynamics. We will discuss some of these examples below. 

  In the present paper we will concentrate on establishing equidistribution results with an effective rate of convergence for certain non-amenable groups,  and in particular  lattice subgroups of semisimple algebraic groups. A major ingredient in our considerations will be mean ergodic theorems for actions of these groups.

We note that a mean ergodic theorem
for spaces with infinite measure is a novel and distinctly non-classical phenomenon. 
Indeed, it is well-known (see \cite[Ch. 2]{Aa}) that for an action of a single
transformation on a space with infinite measure, no formulation of such a result is possible. Likewise, effective rate of orbit equidstribution 
is a phenomenon that does not arise in the ergodic theory of amenable groups, since the ergodic averages may converge arbitrarily slowly. 

\subsection{Historical background}

 The problem of extending ergodic theory to general countable groups was raised half a century ago by Arnol'd and Krylov  \cite{AK}. They have established equidistribution of dense free subgroups of $SO_3(\RR)$ acting on $S^2$, w.r.t. the word length, and have formulated the problem of establishing ergodic theorem for balls w.r.t. word length for actions of general  countable finitely generated groups. Motivated by the problems raised in \cite{AK}, Kazhdan \cite{K} has established that the orbits of certain dense $2$-generator subgroups of the isometry group of the plane satisfy  a ratio ergodic theorem, namely that for every $x\in X$ and any two bounded open sets $A_1$ and $A_2$ (with nice boundary) 
 $$\lim_{t\to \infty}\frac{\abs{\set{\gamma\in B_t \,;\, \gamma^{-1}x\in A_1}}}{\abs{\set{\gamma\in B_t \,;\, \gamma^{-1}x\in A_2}}}=\frac{m(A_1)}{m(A_2)}\,.$$
Here $B_t$ denotes the ball of radius $t$ w.r.t. the word length on the free semigroup, so that the counting is in effect governed by the weights given by convolution powers, and $m$ is Lebesgue measure on the plane.
 Kazhdan has raised in \cite{K} the question of extending this result to other dense subgroups of a Lie group $G$ acting on a homogeneous space $X=G/H$,  particularaly when the action is isometric.
Motivated by both \cite{AK} and \cite{K}  Guivarc'h in \cite{Gu1} has proved the mean ergodic theorem for actions of free groups on a probablity space generalizing von-Neumann's classical result, and in \cite{Gu2} has established a generalization of Kazhdan's ratio ergodic theorem for dense subgroups of  isometries of Euclidean spaces, the weights being given again by the convolution powers of a fixed probability measure on $\Gamma$. Guivarc'h has also raised the problem of establishing equidistribution  results in the generality of actions with a unique invariant measure.

\subsection {Ergodic theorems on homogeneous spaces : Arnold's problems}

Subsequently (see \cite{A} problems 1996-15 p. 115, 2002-16 p. 148) Arnol'd has revisited this topic and raised the following challenges.
Consider the standard Lorentzian form $Q(x,y,z)=x^2+y^2-z^2$ on $\RR^{2,1}$, and the identity component
of the group of isometries of the form, denoted by $G=SO^0(2,1)$. Under the standard linear action of
$G$,  the space $\RR^{2,1}$ decomposes into three invariant subsets of different types, as follows.
\begin{enumerate}
\item  the light cone $\mathcal{C}$, namely the set where the form vanishes, 
 \item the two-sheeted hyperboloid $\mathcal{H}$ given by $\set{x^2+y^2-z^2=1}$, each of whose components inherits a $G$-invariant  Riemannian structure of constant negative curvature isometric to hyperbolic plane.  Thus $\cH$ is a homogeneous space $G/K$ with compact stability group conjugate to $K=SO_2(\RR)\cong \TT$, 
\item the one-sheeted hyperboloid known as the de-Sitter space $\mathcal{S}$ given by $\set{x^2+y^2-z^2=-1}$, which inherits a $G$-invariant two-dimensional Lorentzian structure, and  is a homogeneous space $G/H$ with stability group conjugate to $H=SO^0(1,1)\cong \RR$. 
\end{enumerate}
The group $G$ has a natural action on the projectiviziation of the light cone, namely the usual action by fractional linear transformations of 
the circle. The circle forms the boundary of the hyperbolic plane and is denoted $\mathcal{B}$. 

Consider now any  lattice subgroup $\Gamma\subset G $. Then all orbits of the lattice in hyperbolic space are discrete, and all  orbits of the lattice on the boundary are dense. 
On the de-Sitter space, almost all $\Gamma$-orbits are dense,  but not all. For example, the $\Gamma$-orbit of a point is discrete  if the intersection its stability group with the lattice is a lattice in the stability group. Thus two problems that arise naturally and which were formulated by Arnol'd are 
\begin{enumerate}
\item establish equidistribution  of the lattice orbits on the boundary $\mathcal{B}$.
\item establish ergodic theorems for dense  lattice orbits in the de-Sitter space $\mathcal{S}$.
\end{enumerate}
The first problem was solved in \cite{G2} (see also \cite{gm,go} for generalisations),
and the second problem was solved 
by F.~Maucourant (unpublished) and in greater generality with explicit rate in \cite{GN2}.
The distribution of orbits for the action of $\Gamma$ on the light cone $\mathcal{C}$ was computed 
in \cite{l}, \cite{n}, \cite{G1}, \cite[Sec.~12]{GW}.

\ignore{
As to discrete lattice orbits in the de-Sitter space, it is of course natural to consider the asymptotics of the main term  in the counting function of orbit points in a ball. The latter problem was solved in  \cite{DRS}, with a simpler argument  in \cite{EM}. Effective estimates of the error term in the case of general homogeneous symmetric varieties are developed in  \cite{GN2}. 
}

\subsection{Ergodic theorems beyond amenable groups : some surprises}
We now turn to explicating the results alluded to above and to describing their general context. 

The study of the distribution of $G$-orbits in a space $X$ proceeds by fixing a family of bounded Borel  measures $\beta_t$ on $G$ for  $t\in \ZZ^+$ or $t\in \RR^+$.
The measures $\beta_t$ are not necessarily probability measures. For example, one important special case is when $B_t\subset G$ is a family of bounded sets of positive Haar measure, and we then fix a choice of growth rate function $V(t)$,  and define $\beta_t$ to be Haar measure on $B_t$ divided by $V(t)$. The growth function $V(t)$ may be of lower order of magnitude than $m_G(B_t)$, for example. 

We consider the operators, defined on a compactly supported test-function $f:X\to\mathbb{R}$ by 
$$\pi_X(\beta_t)f(x)=\int_G f(g^{-1}x)d\beta_t(g)\,.$$
Thus in the special case noted above 
$$\pi_X(\beta_t)f(x)=\frac{1}{V(t)}\int_{g\in B_t} f(g^{-1}x)dg\,.$$
 The properties of this family of operators provide the key to analyzing the distribution of the orbit $G\cdot x$. Of course, in the classical case of amenable groups acting by measure-preserving transformations on a probability space we have for a family of sets $B_t\subset G$ :
 \begin{enumerate}
 \item[(i)] the right choice of the growth function $V(t)$ is the total measure $ m_G(B_t)$, so that the operators above become averaging (i.e. Markov) operators, 
 \item[(ii)] the limit of the time averages as $t\to \infty$ is the space average of the function, when the probability measure is ergodic. In particular, the limiting distribution is $G$-invariant. 
\end{enumerate}

It turns out that the ergodic theory of non-amenable groups is full of surprises, and reveals several 
phenomena that have no analogues 
in classical amenable ergodic theory. 
\begin{enumerate}
\item the operators $\pi_X(\beta_t)$ may fail to converge, even when $\beta_t$ are normalized ball averages  w.r.t. a word metric and the action is an isometric action on a compact $G$-transitive space preserving Haar measure.
\item  the operators $\pi_X(\beta_t)$ may converge to a limit operator, but the limit may be different than the ergodic mean, 
even when $\beta_t$ are normalized ball averages  w.r.t. a word metric and the action is an isometric action on a compact $G$-transitive space preserving Haar measure.
\item When the invariant measure is infinite, the operator $\pi_X(\beta_t)$ associated with a family $B_t$ may converge for a choice of growth function $V(t)$ which is different than $m_G(B_t)$, with convergence for almost all points, or even for all points $x$ outside a countable set :
$$\lim_{t\to \infty} \frac{1}{V(t)}\int_{g\in B_t} f(g^{-1}x)dg=\int_X fd\nu_x\,.$$
\item The limit measure $\nu_x$ appearing in (3) may be non-invariant and depend non-trivially on the initial point $x$. Furthermore, the limit measure may be completely different  if the family of sets $B_t$ which are taken as the support of the measures $\beta_t$ is changed. 
\item The expression in (3) may converge for {\it each and every } $x\in X$,  but still the measure $\nu_x$ may be non-invariant, and it may depend on the initial point $x$ and on the family $B_t$. This may happen even when  the invariant measure is unique and even when the action is isometric. 

\item 
The operators $\pi_X(\beta_t)$ in (3) may converge with a {\it uniform rate of convergence}, valid for almost all points, 
or even for all points, namely 
$$\abs{ \frac{1}{V(t)}\int_{g\in B_t} f(g^{-1}x)dg-\int_X fd\nu_x}\le C(x,f) V(t)^{-\delta}\,.$$
This can happen in compact spaces and also non-compact spaces. 

\item  As a result, convergence of the ratios  : $$\frac{\abs{\set{\gamma\in B_t \,;\, \gamma^{-1}x\in A_1}}}{\abs{\set{\gamma\in B_t \,;\, \gamma^{-1}x\in A_2}}}$$ 
may take place at a uniform rate for almost all points. As before,  even for an isometric action (with infinite invariant measure)  the uniform rate may apply to all points,  but  $\nu_x$ appearing in the limiting expression $\frac{\nu_x(A_1)}{\nu_x(A_2)}$ may depend on $x$ and $B_t$.

\end{enumerate}
We remark  that (1) and (2) are implicit already in \cite{AK}, \cite{Gu1}, and that 
(1) has been noted explicitly in \cite{Bew} (see also Theorem \ref{th:free} below).

The phenomena described in (3) and (4) were first demonstrated by a pioneering result of Ledrappier \cite{l} on the distribution 
of orbits of lattice subgroups of $SL_2(\RR)$ in the real plane, see also \cite{lp1} and \cite{lp2}. 
Equivalently, the result applies to the action of a  lattice in $SO^0(2,1)$ on the light cone $\mathcal{C}$ above (see \cite[Th.~12.2]{GW} for more information).

The phenomena described in (5) for  isometric actions has been first demonstrated in  \cite[Cor.~1.4(ii)]{GW} (see also Theorem \ref{th:inf} below).

Regarding (6), note that mean and pointwise ergodic theorems 
with uniform rates for semisimple Kazhdan groups acting on probability spaces
have been established in \cite{n0,MNS,GN1}. The problem whether an ergodic theorem implies
equidistribution with rates for all points forms one of the main subjects of the present paper. The solution to this problem gives the phenomena described in (7) 
as immediate corollaries.

\subsubsection{The mean ergodic theorem and equidistribution in compact spaces }
Assume now that  the space $X$ is equipped with a $G$-invariant  probability  measure $\mu$. 
We say that $\pi_X(\beta_t)$ satisfies
the {\em mean ergodic theorem} in $L^2$  (with limit operator $\mathcal{P}$) if for every $f\in L^2(X,\mu)$,
\begin{equation}\label{eq:mean}
\left\|
\int_G f(g^{-1}x)d\beta_t(g)-\mathcal{P}f(x)\right \|_2\to 0\quad\hbox{as $t\to \infty$,}
\end{equation}
where $\mathcal{P}:L^2(X,\mu)\to L^2(X,\mu)$ is a linear operator,  
which may be different than the orthogonal projection on the space
of $G$-invariant function, see for instance Theorem \ref{th:free} below. 

The mean ergodic theorem is known to hold
for several large classes of lcsc groups, including general amenable groups (see \cite{n1} for a survey) and also for semisimple $S$-algebraic  groups and their lattice subgroups (see \cite{GN1} for a comprehensive discussion).
The next obvious question regarding the distribution of orbits is that of pointwise convergence
of the averages, namely whether for  every $f\in L^2(X,\mu)$ and  almost all $x\in X$ 
\begin{equation}\label{eq:pointwise}
\abs{\int_G f(g^{-1}x)d\beta_t(g)-\mathcal{P}f(x)} \longrightarrow 0\quad\hbox{as $t\to \infty$}.
\end{equation}

When the space $X$ is a compact metric space, one can consider sharpening the pointwise ergodic theorem in two material ways. 
The first is to establish {\it pointwise everywhere} convergence when $f$ is continuous, namely that (\ref{eq:pointwise}) 
holds for every point $x\in X$ without exception, in which case we say that the orbits of $G$ in $X$ are equidistributed.
It was noted in \cite{GN1} (based on an earlier argument due to Guivarc'h \cite{Gu1}) that for isometric actions 
on compact spaces with an invariant ergodic probability measure of  full support,  pointwise everywhere convergence of the averages for continuous functions {\it follows } 
from the mean ergodic theorem. 
This result has as a consequence the fact that such actions 
are in fact uniquely ergodic.  Thus unique ergodicity can be established via spectral arguments.

The second is to establish that the convergence in (\ref{eq:pointwise}) proceeds at a fixed rate, uniformly for every starting point, if the function 
$f$ is H\"older continuous, namely 
\begin{equation}\label{eq:quantitative}
\abs{\int_G f(g^{-1}x)d\beta_t(g)-\mathcal{P}f(x)} \le C(f,x,)E(t)
\end{equation}
where $E(t)\to 0$ as $t\to \infty$. In this case we say that the orbits have a uniform rate of equidistribution, and would like to estimate this rate. 

In the present paper we will establish equidistribution results with an effective uniform rate in two
significant cases, namely isometric and transitive action. To establish 
a quantitative version of these results for actions of general groups $G$ we will require  the spectral 
assumption of the existence of a spectral gap. Recall that a unitary representation has spectral gap if it has no
almost invariant sequence of unit vectors. We emphasize that this assumption is necessary for 
the conclusion to hold, and that its validity is a very common phenomenon.
 For example, all actions of groups with property $T$ have a spectral gap (in the orthogonal complement of the invariants). 

%

We also show that the convergence rate is uniform on the sets 
$$C^a(X)_1=\set{f\in C(X)\,;\, \sup_{x\in X} |f(x)|+ \sup_{x\neq y}\frac{\abs{f(x)-f(y)}}{d(x,y)^a}\le 1}$$
of H\"older continuous  functions
with H\"older norm bounded by one. 

Let us now describe some concrete instances of these results, when the action preserves a probability measure.  

\subsection{Uniform rate of equidistribution : some examples}
\subsubsection{Free groups}

Let $\mathbb{F}_r$ be a free group with $r$ generators, $r\ge 2$.
We denote by $\ell(\gamma)$ the length of an element $\gamma\in\Gamma$ w.r.t. the free 
generating set, and by $B_{2n}$ the ball of radius $2n$. 
We denote by $\vare_0 :\mathbb{F}_r\to\set{\pm 1}$ the sign character 
of the free group, taking the value $1$ on words of even length, and $-1$ on words of odd length. 

Given a unitary 
representation $\pi$ of $\mathbb{F}_r$ on a Hilbert space $\cH$, let $\cH^1$ denote the space of
invariants, and let $\cH^{\vare_0}$ 
denote the space realizing the character $\vare_0$.  Any vector $f_0\in \cH^{\vare_0}$ satisfies
$\pi(\gamma)f_0=(-1)^{\ell(\gamma)}f_0$.

Given a decreasing family of finite index subgroups $\Gamma_i$ of $\FF_r$, we denote
by $\hat \FF_r$ the profinite completion which is equipped with an invariant metric defined by
$$
d(\gamma_1,\gamma_2)=\max\{|\FF_r:\Gamma_i|^{-1}:\, \gamma_1^{-1}\gamma_2\notin \hat \Gamma_i\}
$$
for $\gamma_1,\gamma_2\in \hat \FF_r$.

\begin{theorem}\label{th:free}
\begin{enumerate}
\item Consider an isometric action of $\FF_r$ on a compact manifold $X$. 
Let $\mu$ be an ergodic  smooth probability measure on $X$ with full support  such that the representation of $\mathbb{F}_r$
on  $L^2_0(X,\mu)$ has a spectral gap. Then for every H\"older-continuous $f\in C^a(X)_1$ and every $x\in X$,
\begin{equation}\label{eq:mean_free}
\frac{1}{\#B_{2n}}\sum_{\gamma\in B_{2n}} f(\gamma^{-1}x)= \mathcal{P}f(x)+O\left(e^{-\theta_a n}\right)
\end{equation}
where 
\begin{itemize}
\item $\theta_a>0$ depends only on the spectral gap and $\dim (X)$, and
\item the operator $\mathcal{P}$ is given by 
$$
\mathcal{P}f=\int_X f\,d\mu+\frac{r-1}{r}\left(\int_X f\bar f_0\,d\mu\right)f_0,
$$
where $f_0\equiv 0$ when $\cH^{\vare_0}=0$ which is always the case when $X$ is connected.  Otherwise,
$\cH^{\vare_0}$ is $1$-dimensional, and $f_0$ denotes a unit vector that spans $\cH^{\vare_0}$.
\end{itemize}

\item Let $\Gamma_i\subset \FF_r$ be a decreasing family of finite index subgroups such that
\begin{itemize}
\item $|\Gamma_i:\Gamma_{i+1}|$ is uniformly bounded,
\item the family of representations $L_0^2(\FF_r/\Gamma_i)$ of $\FF_r$ satisfies  property $(\tau)$
(i.e., has uniform spectral gap).
\end{itemize}
Then for every  $f\in C^a(\hat \FF_r)_1$ and $x\in \hat\FF_r$,
\eqref{eq:mean_free} holds as well, with respect to the Haar probability measure $\mu$ on
$\hat \FF_r$, the pro-finite completion associated with the family $\Gamma_i$.
\end{enumerate}
\end{theorem}

We recall that taking $X={S}^2$ to be the unit sphere in $\RR^3$, and $G=SO_3(\RR)$, it was shown
in \cite{LPS,LPS2} (see also \cite{c,o} for generalisations to higher-dimensional spheres) that certain 
dense subgroups $\Gamma\subset G$ admit a spectral gap in their representation  on $L^2_0(S^2)$. The class of such subgroups was 
significantly enlarged recently in \cite{BG}. It may  even be the case that {\it every} dense finitely generated subgroup of $G$ admits a spectral gap, without exception. 
In any case, whenever a spectral gap exists, {\it every} orbit of the dense group become  equidistributed
on the sphere at a uniform rate, depending on the size of the spectral gap (as well as 
the parameters $r$ and $a$, of course).

\subsubsection{Lattices in semisimple algebraic groups}\label{sec:lat}
Let ${\sf G}\subset \hbox{GL}_d$ be a simply connected absolutely simple algebraic group defined over a
local field $K$ of characteristic zero which is isotropic over $K$ (for example, ${\sf G}=SL_d$).
Let $G={\sf G}(K)$, and let $\Gamma$ be a lattice in $G$.
We fix a norm on $\hbox{Mat}_d(K)$ which is the Euclidean norm if $K$ is Archimedean
and the $\max$-norm otherwise. 
Let $B_t$ denote the ball $\set{g\in G\,;\, \log \|g\|<t}$.

\begin{theorem}\label{th:lattice}
\begin{enumerate}
\item Consider an isometric action of $\Gamma$ on a compact manifold $X$.
Let $\mu$ be a smooth probability measure on $X$ with full support such that the action of $\Gamma$
in $L^2_0(X,\mu)$ has a spectral gap. Then for every $f\in C^a(X)_1$ and
every $x\in X$,
\begin{equation}\label{eq:mean_lattice}
\frac{1}{\#(\Gamma \cap B_t)}   \sum_{\gamma\in B_t} f(\gamma^{-1}x)
= \int_X f\,d\mu+O\left(e^{-\theta_a t}\right)
\end{equation}
with explicit $\theta_a>0$.

\item Let $\Gamma_i\subset \Gamma$ be a decreasing family of finite index subgroups such that
\begin{itemize}
\item $|\Gamma_i:\Gamma_{i+1}|$ is uniformly bounded,
\item the family of representations $L_0^2(\Gamma/\Gamma_i)$ of $\Gamma$ satisfies  property $(\tau)$.
\end{itemize}
Then for every  $f\in C^a(\hat \Gamma)_1$ and every $x\in \hat\Gamma$,
\eqref{eq:mean_lattice} holds as well,
with respect to the Haar probability measure $\mu$ on the pro-finite completion $\hat \Gamma$ associated with the family $\Gamma_i$.
\end{enumerate}
\end{theorem}

\begin{remark}
One can also formulate a version of Theorem \ref{th:lattice} for 
general semisimple $S$-arithmetic group, but then one needs to impose the  
condition that the unitary representation of $G^+$ induced from the representation
of $\Gamma$ on $L^2_0(X,\mu)$ has a strong spectral gap (see \cite{GN1} for the terminology). 
\end{remark}

\subsubsection{Transitive actions}
Another significant case where pointwise everywhere convergence holds with a uniform rate is 
that of transitive actions. In this case this property holds for every bounded Borel function.  

Let ${\sf G }\subset \hbox{GL}_d$ be a linear algebraic group defined over a
local field $K$ of characteristic zero, and $G={\sf G}(K)$.
We fix a norm on $\hbox{Mat}_d(K)$ which is the Euclidean norm if $K$ is Archimedean
and the $\max$-norm otherwise. 
Let $\beta_t$ denote the Haar-uniform probability measure
on $G$ supported on $\set{g\in G\,;\,\log \|g\|<t}$.

\begin{theorem}\label{th:trans}
Consider a transitive continuous action of $G$ on a compact space $X$
that supports invariant Borel probability measure $\mu$.
Assume that the Haar-uniform averages $\pi_X(\beta_t)$ satisfy the mean ergodic theorem, namely for $f\in L^2(X,\mu)$, 
$$
E(f,t):=\left\| \pi_X(\beta_{t})f(x)- \int_X f\,d\mu \right\|_2\to 0\quad\hbox{as $t\to\infty$.}
$$
Then for every bounded Borel function $f$ on $X$ with $\sup |f|\le 1$ and for every $x\in X$,
$$
\pi_X(\beta_{t})f(x)= \int_X f\,d\mu+O\left(\left(\sup_{s\in (t-1,t+1)} E(f,s)\right)^{\theta} \right)
$$
with an explicit $\theta\in (0,1)$ independent of $f$ and $x$.
\end{theorem}

Proofs of Theorems \ref{th:free}, \ref{th:lattice}, and \ref{th:trans} will be given 
in Section \ref{sec:last}.

\subsection{Ergodic theorems : spaces with infinite measures}

Let us now turn to spaces with infinite invariant measure, and consider the problem of  
establishing mean and pointwise ergodic theorems and quantitative equidistribution
of orbits for general group actions on such spaces. In general this basic  challenge is largely unexplored 
and here we take up the important case of dense subgroups $\Gamma \subset G$ acting isometrically 
by translations, where we can establish {\it pointwise everywhere} 
convergence with a uniform rate. 
To that end we introduce a natural generalization of the mean ergodic theorem in this setting below (see
Definition \ref{def:mean} below).
In particular,  we obtain the following equidistribution result
that provides a quantitative version of \cite[Cor.~1.4]{GW}.

Let ${\sf G}\subset \hbox{GL}_d$
be a semisimple simply connected algebraic group which is defined over a number field $K$ and is $K$-simple.
Let $T$ and $S$ be finite sets of valuations of $K$ with $T\subset S$.
For $v\in S$, we denote by $K_v$ the corresponding completions.
Let $O_S$ denote the ring of $S$-integers and $\Gamma={\sf G}(O_S)$.

Let 
\begin{align*}
H(g)&=\prod_{v\in S} \|g_v\|_v\quad\hbox{for $g=(g_v)\in \prod_{v\in S} {\sf G}(K_v)$},
\end{align*}
where $\|\cdot\|_v$ are norms on $\hbox{Mat}_d(K_v)$ as in Section \ref{sec:lat}.

\begin{theorem}\label{th:inf}
Assume that $\Gamma$ is dense in $G=\prod_{v\in T} {\sf G}(K_v)$ with respect to the diagonal embedding.
Then there exist $\alpha\in\mathbb{Q}^+$ and $\beta\in\mathbb{N}$  
such that for every H\"older function on $G$
with exponent $a$ and compact support and for every $x\in G$
$$
\frac{1}{t^{\beta-1}e^{\alpha t}} \sum_{\gamma\in \Gamma:\, \log H(\gamma)<t} f(\gamma x)=\int _G f(g)\frac{dm_G(g)}{H(gx)^\alpha}+O_{f,x}(e^{-\theta_a t})
$$
uniformly for $x$ in compact sets, 
where $m_G$ is a suitably normalised Haar measure on $G$ and $\theta_a>0$.
\end{theorem}

We note that the $L^2$-convergence for the operators appearing in Theorem \ref{th:inf} is a special case of the results
of \cite{GN2}. Hence, since the action of $\Gamma$ on $X$ is isometric, Theorem \ref{th:inf} follows
from Theorem \ref{th:isometric}(2) below. 

We refer to \cite[p. 107]{GW} for the identification of $\alpha$, and also for the fact that taking the family $B_t$ associated with the distance function given by a 
power of the height function, will change the power of the density function appearing in the limiting distribution. 
Thus, as we have already mentioned above, this result demonstrates that 
in the infinite-measure setting the limit measure does not have to be invariant and
may depend nontrivially on the initial point $x$ and the family $B_t$, even if the action is an isometric action with a spectral gap.

\subsection{Dense groups of isometries :  Kazhdan's conjecture}

Let $(X,d)$ be an lcsc metric space, and let $G=Isom(X)$ be its group of isometries. 
Assume that the action of $G$ on $X$ is transitive, and let $m_X$ be the unique isometry-invariant 
Radon measure on $X$. Fix two bounded open sets $A_1$ and $A_2$ with boundary of zero measure. 
Consider a countable dense subgroup $\Gamma\subset G$, and a family of sets $B_t\subset \Gamma$, for example, balls w.r.t. a left-invariant 
metric. For each $x\in X$  the orbit $\Gamma\cdot x$ is dense in $X$ and we can form the ratios 
$$\frac{\abs{\set{\gamma\in B_t \,;\, \gamma^{-1}x\in A_1}}}{\abs{\set{\gamma\in B_t \,;\, \gamma^{-1}x\in A_2}}}\,.$$
Consider the problem of whether the ratios satisfy a ratio ergodic theorem, namely whether the limit 
exists as $t\to \infty$ and furthermore whether it is given by 
\begin{equation}\label{eq:ratio}
\lim_{t\to \infty} \frac{\sum_{\gamma\in B_t} \chi_{A_1}(\gamma^{-1}x)}{\sum_{\gamma\in B_t} \chi_{A_2}(\gamma^{-1}x)}=\frac{m_X(A_1)}{m_X(A_2)}.
\end{equation}
This problem was raised by D. Kazhdan in \cite{K}, where the case of certain dense $2$-generator subgroups of $Isom(\RR^2)$ acting on the plane was studied. Assuming  
one of the generators was an irrational rotation, a version of (\ref{eq:ratio}) was established, but with $B_t$ taken as balls in the free group or semigroup, and not as balls w.r.t. the word metric. This amounts to considering weighted averages on $\Gamma$, the weights being given by convolution powers. This result was generalized 
by Y. Guivarc'h  \cite{Gu2} who considered weighted averages given by convolution powers on  dense subgroups of $Isom(\RR^n)$  acting on $\RR^n$ (note also that a gap in the argument in \cite{K} was closed in \cite{Gu2}). For further results in this direction see  \cite{v1}, \cite{v2}. 

Theorem \ref{th:inf} has of course a direct bearing on this problem. In principle, to show that ratios converge, one does not need to establish the much stronger result  that both the numerator and the denominator converge at a common rate and find an explicit expression for the rate. However that is precisely  the conclusion of Theorem \ref{th:inf}, so as an immediate corollary, we obtain the following 
\begin{cor}
Keeping the notation and assumptions of Theorem \ref{th:inf}, we have, 
\begin{enumerate}
\item if $f_1$ and $f_2$ are continuous  and $f_2\ge 0$ (and not identically zero), then
for every $x_1,x_2\in X$,
$$
\lim_{t\to \infty}\frac{\sum_{\gamma\in \Gamma:\, \log H(\gamma)<t} f_1(\gamma x_1) }{\sum_{\gamma\in
    \Gamma:\, \log H(\gamma)<t} f_2(\gamma x_2) }
=\frac{\int _X f_1(g)H(g x_1)^{-\alpha}dm_X(g)}{\int _X f_2(g) H(g x_2)^{-\alpha}dm_X(g)},
$$
\item if in addition $f_1$ and $f_2$ are H\"older continuous with exponent $a$, then
for every $x_1,x_2\in X$,
$$
\frac{\sum_{\gamma\in \Gamma:\,\log  H(\gamma)<t} f_1(\gamma x_1) }{\sum_{\gamma\in \Gamma:\, \log H(\gamma)<t}
  f_2(\gamma x_2) }
=\frac{\int _X f_1(g)H(gx_1)^{-\alpha}dm_X(g)}{\int _X f_2(g) H(gx_2)^{-\alpha}dm_X(g)}
+O_{f_1,f_2, x_1,x_2}(e^{-\theta_a t})
$$
uniformly over $x_1,x_2$ in compact sets.
\end{enumerate}
\end{cor}

Thus the ratios converge for every point, with uniform rate, but the limit is {\it not} the ratio of the integrals with respect to  the isometry invariant measure, but with respect 
to a {\it different} measure. 

We also remark that if $f_1$ and $f_2$ are bounded measurable functions with bounded support, with $f_2\ge 0$ and not identically zero, 
then the ratios converge to the stated limit at almost every point, with uniform rate. This is a consequence of the results of \cite{GN2}. 
 
 
\section{Proof of quantitative equidistribution results}

Let $G$ be an lcsc group acting measurably on a measurable space $X$ equipped 
with a $\sigma$-finite quasi-invariant measure $\mu$. We fix an increasing filtration of $X$ be measurable sets 
$X_r$ with $r\in\mathbb{N}$ of finite measure. 
We denote by $\|\cdot\|_{p,r}$ the $L^p$-norm with respect to the measure $\mu|_{X_r}$. 

We consider families $\beta_t$ of bounded Borel measures on $G$, and in particular, 
given a family of sets $B_t$
on $G$ for $t\ge t_0$, and a positive growth function $V(t)$, we consider the operators:
$$
\pi_X(\beta_t)f(x)=\frac{1}{V(t)}\int_{g\in B_t}f(g^{-1}x)\, dg
$$
for measurable $f:X\to \mathbb{R}$.

\begin{definition}\label{def:mean}
The operators $\pi_X(\beta_t)$ satisfy the {\it mean ergodic theorem} in $L^1$ 
for the action of $G$ on $X$ if for every $r\in \mathbb{N}$ and $f\in L^1(X,\mu|_{X_r})$,
the  sequence $\pi_X(\beta_t)f$  converges in $L^1(X,\mu|_{X_r})$.
\end{definition}

It is clear from the definition that  for $1\le p < \infty$ there exist linear operators
$$
\mathcal{P}_r: L^p(X,\mu|_{X_r})\to  L^p(X,\mu|_{X_r})
$$
such that
\begin{equation}\label{eq:mean-inf}
E_r(f,t):=\left\|\pi_X(\beta_t)f-\mathcal{P}_rf \right\|_{p,r}\to 0\quad\hbox{as $t\to\infty$,}
\end{equation}
and since
$$
\mathcal{P}_{r+1}|_{L^p(X,\mu|_{X_r})}=\mathcal{P}_{r},
$$
it is consistent to denote these operators by $\mathcal{P}$. We shall then say that $\pi_X(\beta_t)$ satisfy the mean ergodic theorem with limit operator $\mathcal{P}$.
\begin{remark}
We note that our notion of the mean ergodic theorem depends on the choice of the filtration $X_r\subset X$
and on the normalisation $V(t)$. It is of course necessary to choose the normalisation so that the operator
$\mathcal{P}$ is not trivial. Then the mean ergodic theorem yields  information about the limiting distribution of the orbits.
\end{remark}

As noted already above, the fact that the foregoing formulation of the mean ergodic theorem
for spaces with {\it infinite} measure is meaningful is an indication of a novel and distinctly non-classical phenomenon. 
Indeed, it is well-known (see \cite[Ch. 2]{Aa}) that for an action of a single
transformation on a space with infinite measure, no normalisation $V(t)$ for which 
(\ref{eq:mean-inf}) holds can be found. Nonetheless, it has been gradually realised that mean ergodic theorems and 
equidistribution results do hold for some classes of action on infinite measure spaces (see \cite{l,lp1,lp2,G1,G2,GW}).
In fact, in the forthcoming paper \cite{GN2} we establish the mean ergodic
theorem for lattices in $S$-algebraic semisimple groups acting on general algebraic homogeneous 
spaces. This result is part of a systematic approach to ergodic theory on homogeneous spaces via the duality 
principle.

\subsection{Isometric actions}
Let  us assume now that $X$ is a metric space equipped with a Radon measure $\mu$, so that the
measures of balls are finite. We fix a filtration of $X$ by balls $X_r$ of radius $r$ centered at some
fixed $x_0\in X$.
We denote by $D_\vare(x)$ the closed ball in $X$ of radius $\vare$ centered at $x$. 
\begin{definition}
\begin{enumerate}
\item We say that the measure $\mu$ is {\em uniformly of full support} if 
for every $r\in\mathbb{N}$ and $\vre\in (0,1]$, 
$$
\inf_{x\in X_r} \mu(D_\vre(x))>0.
$$

\item We say that the measure $\mu$ has {\em local dimension at most $\rho$} if 
for every $r\in\mathbb{N}$, $\vre\in (0,1]$ and $x\in X_r$, 
$$
\mu(D_\vre(x))\ge m_r\vre^\rho.
$$
\end{enumerate}
\end{definition}

\begin{remark}\label{r:uniform}
If the sets $X_r$ are compact, then every measure $\mu$ of full support
is uniformly of full support. Moreover, if $X$ is a compact manifold and $\mu$ is
a smooth measure with full support, then $\mu$ has local dimension at most $\dim (X)$.
\end{remark}

\ignore{

\begin{definition}
\begin{enumerate}
\item We say that the operator $\mathcal{P}$ is {\it regular} if for every $r\in\mathbb{N}$,
it defines a bounded linear operator $\mathcal{P}:UCB(X_r)\to UBC(X)$ where 
$UBC(\cdot)$ denotes the space of uniformly continuous bounded functions.
\item We say that the operator $\mathcal{P}$ is {\it H\"older regular} if for every $r\in\mathbb{N}$,
it defines a bounded linear operator $\mathcal{P}:C^a(X_r)\to C^a(X)$.
\end{enumerate}
\end{definition}
}

The following theorem is our main technical result concerning equidistribution for isometric actions.

\begin{theorem}\label{th:isometric}
Consider an isometric action of an lcsc group $G$ on the lcsc metric space $X$ equipped with
a quasi-invariant Radon measure $\mu$.
Assume that the mean ergodic theorem holds for the operators $\pi_X(\beta_t)$ in the action of $G$ on $X$.
Then 
\begin{enumerate}
\item[1.] If the measure $\mu$ is uniformly of full support,
then for every uniformly continuous function $f$
such that $\supp(f)\subset X_r$ and $(\mathcal{P}f)|_{X_r}$ is uniformly continuous, we have 
$$
\max_{x\in X_r} \left|\pi_X(\beta_t)f(x)-\mathcal{P}f(x)\right|=o_{r,f}(1)
$$
as $t\to\infty$.

\item[2.]
If the measure $\mu$ has local dimension at most $\rho$,
then for every $f\in C^a(X)_1$ such that $\supp(f)\subset X_r$ and $(\mathcal{P}f)|_{X_r}\in C^a(X)_1$, we have 
$$ 
\max_{x\in X_r} \left|\pi_X(\beta_t)f(x)-\mathcal{P}f(x)\right|\ll_r  E_r(f,t)^{a/(a+\rho)} 
$$
for all sufficiently large $t$.
\end{enumerate}
\end{theorem}

\begin{remark} Let us note two other general approaches that derive equidistribution results
from estimates on $L^2$-norms. 
The first approach \cite{Gu1,CO,GN1}, which is originally due to Guivarc'h,
applies only in the case of compact spaces and does not produce a rate of convergence. 
The second approach \cite[\S8]{CU} uses the theory of elliptic operators,
so that it can only be applied in the setting of Lie groups and sufficiently smooth functions.
\end{remark}

\subsection{Transitive actions}

As noted above, in general the behavior of the operators $\pi_X(\beta_t)$ may
depend quite sensitively of the initial point. Nonetheless, when the action is transitive it is still possible to obtain 
a uniform result, provided that a certain regularity property of the measures $\beta_t$ holds. 

Let $d$ be a right  invariant metric on $G$ compatible with the topology of $G$
such that the closed balls with respect to $d$ are compact
(such a metric always exists, see for instance \cite{hp}).
We denote the closed ball of radius $\vare$ centered at $g\in G$ by $\mathcal{O}_\vare(g)$.
  
\begin{definition}\label{def:holder}
The family of measures $\beta_t$ is called {\em coarsely monotone} if there exists 
monotone functions $\kappa:(0,1]\to (0,\infty)$ and $\delta:(0,1]\to (1,\infty)$
 such that 
$$
\delta_\vre\to 1\;\;\hbox{and}\;\;\kappa_\vre\to 0\quad\quad\hbox{as $\vre\to 0^+$},
$$
and for every $\vre\in (0,1]$, $g\in \mathcal{O}_\vre(e)$, and $t\ge t_0$,
$$
g\cdot\beta_t\le \delta_\vre\, \beta_{t+\kappa_\vre}.
$$
If, in addition, we have $\delta_\vre=1+O(\vre^{a_0})$ for some $a_0>0$, then 
the family of measures is called {\em H\"older coarsely monotone} with exponent $a_0$.
\end{definition}

Let $X$ be a lcsc space on 
which the group $G$ acts transitively and continuously. 
Since $X$ is locally compact, the topology on $X$ coincides with the topology
defined on $X$ by viewing it as a factor space $G$.
We equip $X$ with a $G$-quasi-invariant Radon measure $\mu$.
The space $X$ is equipped with the natural metric (see \cite[\S8]{HR}), which is defined by
\begin{equation}\label{eq:met}
d(x_1,x_2)=\inf\{d(g_1,g_2):\, g_1,g_2\in G,\, g_1 x_0=x_1,\, g_2 x_0=x_2\}.
\end{equation}
where $x_0$ is a fixed element of $X$. We use the filtration on $X$ such that $X_r$ are balls of radius $r$ in $X$ centered 
at some fixed $x_0\in X$.

\begin{theorem}\label{th:transitive}

Assume that the mean ergodic theorem holds for the family of operators $\pi_X(\beta_t)$  in the transitive $G$-action on $X$. 
\begin{enumerate}
\item If $\beta_t$  is a coarsely monotone family of measures, 
then for every bounded Borel function $f:X\to \mathbb{R}$ such that $\supp(f)\subset X_r$
and $\mathcal{P}f$ is uniformly continuous on $X_r$, we have
$$
\max_{x\in X_r} \left|\pi_X(\beta_t)f(x)-\mathcal{P}f(x)\right|=o_{r,f}(1)
$$
as $t\to\infty$.

\item
If in addition $\beta_t$ is H\"older  coarsely monotone with exponent $a_0$
and $\mu$ has local dimension at most $\rho$,
then for every bounded Borel function $f:X\to \mathbb{R}$ such that $\supp(f)\subset X_r$,
and $(\mathcal{P}f)|_{X_r}\in C^a(X)_1$, we have
$$
\max_{x\in X_r} \left|\pi_X(\beta_t)f(x)-\mathcal{P}f(x) \right|\ll_r \left(\sup_{s\in (t-\kappa_1,t+\kappa_1)} E_r(f,s)\right)^\frac{\min(a_0,a)}{\min(a_0,a)+\rho}
$$
for all sufficiently large $t$.
\end{enumerate}
\end{theorem}

\section{Proof of equidistribution for isometric actions}

In this section we prove Theorem \ref{th:isometric}.
We start the proof with the following

\begin{lemma}\label{l:est}
Assume that the mean ergodic theorem holds for the 
family of operators $\pi_X(\beta_t)$. We assume that the action of $G$ on the space $X$ is isometric and 
equipped with a quasi-invariant Radon measure $\mu$ which is uniformly of full support.
Then for all sufficiently large $t$, $r\in \mathbb{N}$, and $y\in X_r$,
$$
\beta_t\left(\left\{g\in G:\, g^{-1}y\in X_r\right\}\right)=O_r(1).
$$
\end{lemma}

\begin{proof}
It follows from the mean ergodic theorem that
$$
\int_{X_r} \left|\beta_t\left(\left\{g\in G:\, g^{-1}x\in
      X_r\right\}\right)-\mathcal{P}\chi_{X_r}(x)\right|d\mu(x)=o_r(1)
$$
as $t\to\infty$, and hence,
$$
c_r(t):=\int_{X_r} \beta_t\left(\left\{g\in G:\, g^{-1}x\in
      X_r\right\}\right)d\mu(x)=O_r(1).
$$
Let $\delta>0$. We observe that for the set 
$$
\Omega_r(\delta,t):=\left\{x\in X_r:\, \beta_t\left(\left\{g\in G:\, g^{-1}x\in
      X_r\right\}\right)>\delta\right\},
$$
we have 
$$
\mu(\Omega_r(\delta,t))\le c_r(t) /\delta.
$$
Therefore, if we choose $\delta=2 c_r(t) /m_{r-1}$ where 
$$
m_{r-1}:=\inf_{y\in X_{r-1}} \mu(D_1(y))>0,
$$
then 
$$
\mu(\Omega_r(\delta,t))<\mu(D_1(y))\quad \hbox{for all $y\in X_{r-1}$.}
$$
Hence, for every $y\in X_{r-1}$ there exists $x\in D_1(y)\subset X_r$ such that $x\notin
\Omega_r(\delta,t)$, i.e., 
$$
\beta_t\left(\left\{g\in G:\, g^{-1}x\in
      X_r\right\}\right)\le \delta=O_r(1).
$$
Since the action is isometric, we have $d(x_0,g^{-1}x)\le d(x_0,g^{-1}y)+1$.
Hence, if $g^{-1}y\in X_{r-1}$, then $g^{-1}x\in X_{r}$, so that
\begin{align*}
\beta_t\left(\left\{g\in G:\, g^{-1}y\in
      X_{r-1}\right\}\right)\le
\beta_t\left(\left\{g\in G:\, g^{-1}x\in
      X_{r}\right\}\right).
\end{align*}
This implies the claim.
\end{proof}

\begin{proof}[Proof of Theorem \ref{th:isometric}]
In the proof we shall use parameters $\vre\in (0,1)$ and $\delta>0$ that will be specified later.
Let
\begin{equation}\label{eq:omega}
\Omega_r(\delta,t)=\left\{x\in X_r:\, \left|\pi_X(\beta_t)f(x)-\mathcal{P}f(x) \right|>\delta
  \right\}.
\end{equation}
Then 
$$
\mu\left(\Omega_r(\delta,t)\right)\le E_r(f,t)/\delta.
$$
Hence, if we assume that 
\begin{equation}\label{eq:delta}
m_{r-1}(\vre):=\inf_{y\in X_{r-1}} \mu(D_\vre(y))>E_r(f,t)/\delta,
\end{equation}
then for every $y\in X_{r-1}$ there exists $x\in D_\vre(y)\subset X_r$ such that $x\notin
\Omega_r(\delta,t)$, i.e., 
\begin{equation}\label{eq:bbb0}
\left|\pi_X(\beta_t)f(x)-\mathcal{P}f(x) \right|\le \delta.
\end{equation}
Let
\begin{equation}\label{eq:om}
\omega_{r}(f,\vre)=\sup\{|f(z)-f(w)|:\,\, z,w\in X_r,\, d(z,w)<\vre\}.
\end{equation}
Since $f$ is uniformly continuous, $\omega_{r}(f,\vre)\to 0$ as $\vre\to 0^+$.
Using that the action of $G$ on $X$ is isometric and $\supp(f)\subset X_r$,
we deduce that
\begin{align*}
&\left|\pi_X(\beta_t)f(x)-\pi_X(\beta_t)f(y)\right|\\
\le\,&   \omega_{r}(f,\vre)\,
\beta_t\left(\left\{g\in G:\, g^{-1}x\in X_r\;\hbox{or}\;g^{-1}y\in X_r\right\}\right)
\ll_r\,   \omega_{r}(f,\vre),
\end{align*}
where the last inequality follows from Lemma \ref{l:est}.
Hence, it follows from (\ref{eq:bbb0}) that for every $y\in X_{r-1}$,
$$
\left|\pi_X(\beta_t)f(y)-\mathcal{P}f(y) \right|\ll_r  \delta + \omega_{r}(f,\vre)+\omega_{r}(\mathcal{P}f,\vre).
$$
This estimate holds provided that $\delta$ satisfies (\ref{eq:delta}). Hence, it follows that for every $r\in\NN$ 
$$
\max_{y\in X_{r-1}} \left|\pi_X(\beta_t)f(y)-\mathcal{P}f(y)\right|\ll_r \tilde E_r(f,t)
$$
where 
$$
\tilde E_r(f,t)=\inf_{\vre\in (0,1)}\Large(E_r(f,t)/m_{r-1}(\vre) + \omega_{r}(f,\vre)+\omega_{r}(\mathcal{P}f,\vre)\Large).
$$
Using that $E_r(f,t)\to 0$ as $t\to \infty$ and $\omega_{r}(f,\vre),\omega_{r}(\mathcal{P}f,\vre)\to 0$
as $\vre\to 0^+$, we conclude that $\tilde E_r(f,t)=o_{r,f}(1)$ as well.
This proves the first part of the theorem.

To prove the second part of the theorem, we observe that under the additional assumptions,
$$
\tilde E_r(f,t)\ll_r \inf_{\vre\in (0,1)} \Large(\vre^{-\rho} E_r(f,t) + \vre^a\Large).
$$
We therefore take $\vre=E_r(f,t)^{1/(a+\rho)}$, and note that since $E_r(f,t)\to 0$ as $t\to\infty$,
we have $\vre\in (0, 1)$ for all sufficiently large $t$. Then it follows that
$$
\tilde E_r(f,t)\ll_r E_r(f,t)^{a/(a+\rho)},
$$
as required.
\end{proof}

\section{Proof of equidistribution for transitive actions}

Our proof of Theorem \ref{th:transitive} follows the same strategy as the proof of Theorem
\ref{th:isometric}. We start with the following lemma  establishing directly that in the transitive
case the quasi-invariant measure is uniformly 
of full support. We use the metric $d$ on the homogeneous space $X$ defined in (\ref{eq:met}). 

\begin{lemma}
For every $r\in\mathbb{N}$ and $\vre>0$,
$$
m_r(\vre):=\inf_{x\in X_r} \mu(D_\vre(x))>0.
$$
\end{lemma}

\begin{proof}
It follows from the definition of the metric on $X$ that
$$
D_s(x)=\mathcal{O}_s(e)\cdot x\quad\hbox{for every $s>0$ and $x\in X$. }
$$
Since the balls $\mathcal{O}_r(e)$ are compact,
there exists $\vre'=\vre'(\vre,r)>0$ such that
$g\mathcal{O}_{\vre'}(e)g^{-1}\subset \mathcal{O}_\vre(e)$ for every $g\in \mathcal{O}_r(e)$. Then 
since the measure $\mu$ is quasi-invariant, for every $g\in \mathcal{O}_r(e)$,
$$
\mu(D_\vre(gx_0))\ge \mu(\mathcal{O}_\vre(e)gx_0)\ge \mu(g \mathcal{O}_{\vre'}(e)x_0)\gg_r \mu(\mathcal{O}_{\vre'}(e)x_0).
$$ 
This proves the claim.
\end{proof}
  
\ignore{

\begin{lemma}\label{l:est2}
Assume that the mean ergodic theorem holds for the 
family of measures $\beta_t$ on $G$ and a transitive continuous action of $G$ on the space $X$
equipped with a Radon measure $\mu$ which is uniformly of full support.
Then for all sufficiently large $t$, $r\in \mathbb{N}$, and $y\in X_r$,
$$
\beta_t\left(\left\{g\in G:\, g^{-1}y\in X_r\right\}\right)=O_r(1).
$$
\end{lemma}

\begin{proof}
As in the proof of Lemma \ref{l:est}, we introduce the sets $\Omega_r(\delta,t)$
(see (\ref{eq:omega})) and observe that $\mu(\Omega_r(\delta,t))\le c_r(t)/\delta$.
Therefore, if we choose $\delta=2 c_r(t)/m_{r-1}(1)$, then 
$$
\mu(\Omega_r(\delta,t))<\mu(D_1(y))\quad \hbox{for all $y\in X_{r-1}$.}
$$
Hence, for every $y\in X_{r-1}$ there exists $x\in D_1(y)\subset X_r$ such that $x\notin
\Omega_r(\delta,t)$. Then
$$
\beta_t\left(\left\{g\in G:\, g^{-1}x\in
      X_r\right\}\right)\le \delta=O_r(1).
$$
Since $y=h^{-1}x$ for some $h\in \mathcal{O}_1(e)$ and $X_{r-1}\subset X_r$, we have
\begin{align*}
\beta_t\left(\left\{g\in G:\, g^{-1}y\in X_{r-1}\right\}\right)&\le
(h\cdot \beta_t)\left(\left\{g\in G:\, g^{-1}x\in X_r\right\}\right)\\
&\le \delta_1\, \beta_{t+\kappa_1}\left(\left\{g\in G:\, g^{-1}x\in X_r\right\}\right)=O_r(1),
\end{align*}
as required.
\end{proof}

}

\begin{proof}[Proof of Theorem \ref{th:transitive}]
Note that, without loss of generality, we may assume that 
$$
E_r(-f,t)=E_r(f,t).
$$
Hence, it suffices to prove the estimate for nonnegative $f$.

Let $\vre\in (0,1)$ and $\delta>0$.
As in the proof of Theorem \ref{th:isometric}, we introduce the set
$$
\Omega_r(\delta,t)=\left\{x\in X_r:\, 
\left|\pi_X(\beta_t)f(x)-\mathcal{P}f(x) \right|>\delta
\right\}
$$
and observe that 
$$
\mu\left(\Omega_r(\delta,t)\right)\le E_r(f,t)/\delta.
$$
Let $\bar E_r(f,t):=\sup_{s\in (t-\kappa_1,t+\kappa_1)} E_r(f,s)$, and
\begin{equation}\label{eq:d}
\delta>\bar E_r(f,t)/m_{r-1}(\vre).
\end{equation}
Then for every $y\in X_{r-1}$ and $s\in (t-\kappa_1,t+\kappa_1)$,
there exists $x_s\in D_\vre(y)\subset X_r$ such that $x_s\notin
\Omega_r(\delta,s)$, i.e., 
\begin{equation}\label{eq:bbb}
\left|\pi_X(\beta_s)f(x_s)-\mathcal{P}f(x_s) \right|\le \delta.
\end{equation}
We set $x_1=x_{t-\kappa_\epsilon}$ and $x_2=x_{t+\kappa_\epsilon}$.
Then since $y=h x_1$ for some $h\in \mathcal{O}_\vre(e)$, it follows from the coarsely monotone property of
$\beta_t$ that
$$
\pi_X(\beta_t)f(y)=\pi_X(h\cdot
\beta_{t})f(x_1)\ge \delta_\vre^{-1}\,\pi_X(\beta_{t-\kappa_\vre})f(x_1),
$$
and 
\begin{align*}
\left|\pi_X(\beta_t)f(y)-\mathcal{P}f(y) \right|\le & \left(\pi_X(\beta_{t})f(y)-
  \delta_\vre^{-1}\,\pi_X(\beta_{t-\kappa_\vre})f(x_1)\right)\\
&+\left|\delta_\vre^{-1}\,\pi_X(\beta_{t-\kappa_\vre})f(x_1)-\mathcal{P}f(y) \right|.
\end{align*}
Similarly, 
$$
\pi_X(\beta_t)f(y) \le \delta_\vre\,\pi_X(\beta_{t+\kappa_\vre})f(x_2),
$$
and hence 
\begin{align*}
\left|\pi_X(\beta_t)f(y)-\mathcal{P}f(y) \right| \le & \left(\delta_\vre\,\pi_X(\beta_{t+\kappa_\vre})f(x_2)-
  \delta_\vre^{-1}\,\pi_X(\beta_{t-\kappa_\vre})f(x_1)\right)\\
&+ \left|\delta_\vre^{-1}\,\pi_X(\beta_{t-\kappa_\vre})f(x_1)-\mathcal{P}f(y) \right|.
\end{align*}
Now we estimate each of the above terms separately.

It follows from (\ref{eq:bbb}) and uniform continuity of $\mathcal{P}f$ on $X_r$ that
\begin{align*}
&|\delta_\vre\,\pi_X(\beta_{t+\kappa_\vre})f(x_2)-
  \delta_\vre^{-1}\,\pi_X(\beta_{t-\kappa_\vre})f(x_1)|\\
\le &\, \delta_\vre\,|\pi_X(\beta_{t+\kappa_\vre})f(x_2)-\mathcal{P}f(x_2)|
+|\delta_\vre \mathcal{P}f(x_2)-\delta^{-1}_\vre \mathcal{P}f(x_1)|\\
 & + \delta_\vre^{-1}\,|\pi_X(\beta_{t-\kappa_\vre})f(x_1)-\mathcal{P}f(x_1)|\\
\le & \, ( \delta_\vre+ \delta_\vre^{-1})\delta+ \delta_\epsilon |\mathcal{P}f(x_2)-\mathcal{P}f(x_1)| +
(\delta_\epsilon-\delta_\epsilon^{-1})|\mathcal{P}f(x_1)|\\
\ll_r& \, \delta +\omega_r(\mathcal{P}f,2\vre)+(\delta_\epsilon-\delta_\epsilon^{-1}),
\end{align*}
where the function $\omega_r$ is defined as in (\ref{eq:om}). Also,
\begin{align*}
&\left|\delta_\vre^{-1}\pi_X(\beta_{t-\kappa_\vre})f(x_1)-\mathcal{P}f(y) \right|\\
\le &\,  \delta_\vre^{-1}\,|\pi_X(\beta_{t-\kappa_\vre})f(x_1)-\mathcal{P}f(x_1)|+\left|\delta_\vre^{-1}\,\mathcal{P}f(x_1)-  \mathcal{P}f(y) \right|\\
\le &\, \delta_\vre^{-1}\delta+\delta_\vre^{-1}|\mathcal{P}f(x_1)-\mathcal{P}f(y)|
+(1- \delta_\vre^{-1})\left|\mathcal{P}f(y) \right|\\
\ll_r &\, \delta +\omega_r(\mathcal{P}f,\vre)+(1-\delta_\epsilon^{-1}).
\end{align*}
Therefore, we conclude that 
\begin{align*}
\left|\pi_X(\beta_t)f(y)-\mathcal{P}f(y) \right|\ll_r 
\delta+ \omega_r(\mathcal{P}f,2\vre)+\delta_\vre -2\delta_\vre^{-1}+1 .
\end{align*}
Since this estimate holds for all $\epsilon\in (0,1)$, $y\in X_{r-1}$ and $\delta$ satisfying (\ref{eq:d}), 
we have 
$$
\max_{y\in X_{r-1}} \left|\pi_X(\beta_t)f(y)-\mathcal{P}f(y)\right|\ll_r \tilde E_r(f,t)
$$
where 
$$
\tilde E_r(f,t)=\inf_{\vre\in (0,1)} \Large\{\bar E_r(f,t)/m_{r-1}(\vre)+ \omega_r(\mathcal{P}f,2\vre)+\delta_\vre -2\delta_\vre^{-1}+1  \Large\}.
$$
Since $\bar E_r(f,t)\to 0$ as $t\to \infty$ and $\delta_\vre\to 1$, $\omega_r(\mathcal{P}f,2\vre)\to 0$ as $\vre\to 0^+$,
it follows that $\tilde E_r(f,t)\to 0$ as $t\to \infty$ too. This implies the first part of the theorem.

To prove the second part of the theorem, we observe that
under the additional assumptions,
$$
\tilde E_r(f,t)\ll_r \inf_{\vre\in (0,1)} \Large(\vre^{-\rho} \bar E_r(f,t) + \vre^{\min(a_0,a)}\Large).
$$
Since $E_r(f,t)\to 0$ as $t\to \infty$, it follows that $\bar E_r(f,t)\in (0,1)$ for sufficiently large $t$.
Taking $\vre=\bar E_r(f,t)^{1/(\min(a_0,a)+\rho)}$, we deduce the second claim.
\end{proof}

\section{Completion of the proofs}\label{sec:last}

\begin{proof}[Proof of Theorem \ref{th:free}]
We deduce Theorem \ref{th:free} from Theorem \ref{th:isometric}(2). 
We recall that the mean ergodic theorem for the free group $\FF_r$ was established in \cite{Gu1,n00}.
Moreover, under the spectral gap assumption, the method of the proof of \cite[Th.~1]{n00} implies that 
$$
\left\|\frac{1}{\#B_{2n}}\sum_{\gamma\in B_{2n}} f(\gamma^{-1}x)-\mathcal{P}f(x)\right\|_2
=O\left(e^{-\theta n}\|f\|_2\right)
$$
for some $\theta>0$ determined by the spectral gap.

Let $G$ be the closure of $\FF_r$ in the isometry group of $X$.
Then the measure $\mu$ is invariant and ergodic with respect to $G$.
Since $X$ is compact, $G$ is compact, and it follows that $\mu$ is supported
on a single orbit of $G$. Hence, $G$ acts transitively on $X$.
Let $G_0$ be the closure in $G$ of the subgroup of $\FF_r$ generated by
the words of even length. Since $G_0$ has index at most two in $G$, the subgroup $G_0$ is open
in $G$, and $X$ consists of at most two open orbits of $G_0$.
This implies that $L^2(X)^{\epsilon_0}$ has dimension at most one and is trivial when
$X$ is connected. Moreover, it is clear that $f_0$ is locally constant
and, in particular, $\mathcal{P}f\in C^a(X)_1$.

Finally, we note that in case (1) the measure $\mu$ has local dimension at most $\dim(X)$ (cf. Remark
\ref{r:uniform}).
In case (2), we have 
$$
\mu(D_\epsilon(x))=\mu(D_\epsilon(e))=\epsilon
$$
when $\epsilon=|\Gamma:\Gamma_i|^{-1}$. Since $|\Gamma_i:\Gamma_{i+1}|$ is uniformly bounded,
this implies that $\mu$ has local dimension at most one.
Now Theorem \ref{th:free} follows from Theorem \ref{th:isometric}(2). 
\end{proof}

\begin{proof}[Prof of Theorem \ref{th:lattice}]
We note that $L^2$-convergence for (\ref{eq:mean_lattice}) with exponential rate
follows from the results of \cite{GN1}.
Indeed, the balls $B_t$ are H\"older admissible by \cite[Ch.~7]{GN1}.
Since in both cases we have a lower estimate on the local dimension
(see the proof of Theorem \ref{th:free}),
Theorem \ref{th:lattice} follows from Theorem \ref{th:isometric}(2).
\end{proof}

\begin{proof}[Proof of Theorem \ref{th:trans}]
It follows from \cite[Ch.~7]{GN1} that the family of measures $\beta_t$
is H\"older coarsely monotone. Hence, Theorem \ref{th:trans} is a consequence of 
Theorem \ref{th:transitive}(2).
\end{proof}

\end{document}